\documentclass[12pt]{article}
\usepackage{amsmath,amsthm,amssymb}
\usepackage[left=3cm,right=3cm, top=2.5cm,bottom=2.5cm,bindingoffset=0cm]{geometry}

\usepackage{amscd}

\usepackage{hyperref}

\usepackage{mathtools}

\usepackage[all,cmtip]{xy}

\usepackage{MnSymbol} 

\usepackage[table]{xcolor}

\makeatletter
\def\th@plain{%
  \thm@notefont{}
  \itshape 
}
\def\th@definition{%
  \thm@notefont{}
  \normalfont 
}
\makeatother

\newtheorem{lemma}{Lemma}[section]
\newtheorem{proposition}[lemma]{Proposition}

\newtheorem{remark-definition}[lemma]{Remark-Definition}
\newtheorem{theorem}[lemma]{Theorem}
\newtheorem{corollary}[lemma]{Corollary}

\newtheorem{proposition-conjecture}[lemma]{Proposition-conjecture}

\theoremstyle{definition}

\newtheorem{definition}[lemma]{Definition}
\newtheorem{remark}[lemma]{Remark}

\begin{document}

\title{Integrability of bi-Hamiltonian systems using Casimir functions and characteristic polynomials}
\author{I.\,K.~Kozlov\thanks{No Affiliation, Moscow, Russia. E-mail: {\tt ikozlov90@gmail.com} }
}
\date{}

\maketitle

\begin{abstract} In this paper we prove that for a pencil of compatible Poisson brackets $\mathcal{P} = \left\{\mathcal{A} + \lambda\mathcal{B} \right\}$ the local Casimir functions of Poisson brackets $\mathcal{A} + \lambda \mathcal{B}$  and coefficients of the characteristic polynomial $p_{\mathcal{P}}$ commute w.r.t. all Poisson brackets of the pencil $\mathcal{P}$. We give a criterion when this family of functions is complete. These results generalize previous constructions of complete commutative subalgebras in the symmetric algebra $S(\mathfrak{g})$ of a finite-dimensional Lie algebra $\mathfrak{g}$ by  A.\,S.~Mishchenko \& A.\,T.~Fomenko,  A.\,V.~Bolsinov \& P.~Zhang and A.\,M.~Izosimov. \end{abstract}

\tableofcontents

\section{Introduction}

It is well-known that bi-Hamiltonian structure and integrability of many systems in physics and mechanics are closely related. Bi-Hamiltonian structures often allow us to construct a natural family of functions in involution. When this family is complete, we obtain an integrable Hamiltonian system. This connection was first observed for infinite-dimensional systems in the pioneering work of F.~Magri 
\cite{Magri78} (this idea was further developed in \cite{Gelfand79}, \cite{Magri84} and \cite{Reiman80}). For finite-dimensional systems on Lie algebras, the fundamental result is the \textbf{argument shift method} by A.\,S.~Mishchenko \& A.\,T.~Fomenko \cite{ArgShift}. It is a generalization of the S.\,V.~Manakov's construction \cite{Manakov76} for the Lie algebra $\operatorname{so}(n)$. The argument shift method has been further developed, leading to several new methods for constructing commutative subalgebras within the symmetric algebra $S(\mathfrak{g})$ of a Lie algebra $\mathfrak{g}$. Here's a brief outline of some milestones in this progress: 

\begin{enumerate}

\item The original argument shift method required polynomial invariants of coadjoint representation. A.V. Brailov's modification (\cite{Bolsinov14}) overcomes this limitation. His approach allows us to construct a commutative \textbf{algebra of polynomial shifts}
$\mathcal{F}_a$ even in the cases when the invariants are not polynomials (see e.g. \cite[Theorem 1]{BolsZhang}).

\item The completeness criterion for the subalgebras $\mathcal{F}_a$  was found by A.\,V.~Bolsinov (\cite{Bolsinov92}).  Later in  \cite{BolsZhang}  A.\,V.~Bolsinov introduced Jordan--Kronecker invariants of a Lie algebra $\mathfrak{g}$. It was proved that \textit{algebra of polynomial shifts $\mathcal{F}_a$ is complete if and only if $\mathfrak{g}$ is of Kronecker type} (see  \cite[Theorem 3]{BolsZhang}). 

\item In \cite{Izosimov14} A.\,M.~Izosimov introduced \textbf{extended Mischenko-Fomenko subalgebras} $\tilde{\mathcal{F}}_a$ and gave a completeness criterion for them. In terms of Jordan--Kronecker invariants that construction was described in   \cite[Section 7]{BolsZhang}. 

\item Then, in  \cite{Kozlov23Shift} the algebra $\tilde{\mathcal{F}}_a$ was extended further to the  \textbf{algebra of shift of semi-invariants} $\mathcal{F}_{a}^{\textrm{si}}$. It was proved that $\mathcal{F}_{a}^{\textrm{si}}$ is complete if and only if $\tilde{\mathcal{F}}_a$ is complete. 

\end{enumerate}

The cornerstone for constructing commutative subalgebras  $\mathcal{F}_a$, $\tilde{\mathcal{F}}_a$ and $\mathcal{F}_{a}^{\textrm{si}}$ for a Lie algebra $\mathfrak{g}$ is the the existence of a pencil of compatible Poisson brackets on the dual space  $\mathfrak{g}^*$. This paper generalizes the results from \cite{BolsZhang}, \cite{Izosimov14}, \cite{Kozlov23Shift} by presenting a similar construction for manifolds $M$ equipped with a pencil of compatible Poisson brackets $\mathcal{P} = \left\{\mathcal{A} + \lambda\mathcal{B} \right\} $:

\begin{enumerate}

\item In Theorem~\ref{Th:CasEigenComm} we show that local Casimir functions of regular Poisson brackets $\mathcal{A} + \lambda \mathcal{B}$  and coefficients of the characteristic polynomial $p_{\mathcal{P}}$ are in involution w.r.t. all Poisson brackets $\mathcal{A} + \lambda \mathcal{B}$.

\item In Theorem~\ref{Th:CompCrit} we give a completeness criterion for the distribution "spanned" by the Casimir functions and the coefficients of $p_{\mathcal{P}}$. 

\end{enumerate}

In Section~\ref{S:LieAlg}  we illustrate the application of Theorems~\ref{Th:CasEigenComm} and \ref{Th:CompCrit} for  compatible Poisson brackets on the dual space of a (finite-dimensional) Lie algebra $\mathfrak{g}$. We demonstrate how these theorems recover previously obtained results about $\mathcal{F}_a$, $\tilde{\mathcal{F}}_a$ and $\mathcal{F}_{a}^{\textrm{si}}$.

While Theorems~\ref{Th:CasEigenComm} and \ref{Th:CompCrit} were more or less well-known to experts in the field (see e.g. \cite[Proposition 2.1]{Izosimov14}), they lacked a formal proof in the literature. This article addresses this gap by providing rigorous proofs of these theorems.

\textbf{Conventions.} All manifolds (functions, Poisson brackets, etc)  are  either real $C^{\infty}$-smooth or complex analytic. Some property holds ``almost everywhere'' or ``at a generic point'' of a manifold $M$ if it holds on an open dense subset of $M$.  We denote $\bar{\mathbb{C}} = \mathbb{C} \cup \left\{ \infty \right\}$. 

\par\medskip

The author would like to thank A.\,V.~Bolsinov and A.\,M.~Izosimov for useful comments. 

\section{Basic definitions}

\subsection{Jordan--Kronecker theorem} 

First, let us recall the canonical form for a pair of skew-symmetric forms. This theorem, which we call the Jordan--Kronecker theorem, is a classical result that goes back to Weierstrass and Kronecker.  A proof of it can be found in \cite{Thompson}, which is based on \cite{Gantmaher88}.

\begin{theorem}[Jordan--Kronecker theorem]\label{T:Jordan-Kronecker_theorem}
Let $A$ and $B$ be skew-symmetric bilinear forms on a
finite-dimension vector space $V$ over a field $\mathbb{K}$ with $\textmd{char }  \mathbb{K} =0$. If the field $\mathbb{K}$
is algebraically closed, then there exists a basis of the space $V$
such that the matrices of both forms $A$ and $B$ are block-diagonal
matrices:

{\footnotesize
$$
A =
\begin{pmatrix}
A_1 &     &        &      \\
    & A_2 &        &      \\
    &     & \ddots &      \\
    &     &        & A_k  \\
\end{pmatrix}
\quad  B=
\begin{pmatrix}
B_1 &     &        &      \\
    & B_2 &        &      \\
    &     & \ddots &      \\
    &     &        & B_k  \\
\end{pmatrix}
$$
}

where each pair of corresponding blocks $A_i$ and $B_i$ is one of
the following:

\begin{itemize}

\item Jordan block with eigenvalue $\lambda_i \in \mathbb{K}$: {\scriptsize  \begin{equation} \label{Eq:JordBlockL} A_i =\left(
\begin{array}{c|c}
  0 & \begin{matrix}
   \lambda_i &1&        & \\
      & \lambda_i & \ddots &     \\
      &        & \ddots & 1  \\
      &        &        & \lambda_i   \\
    \end{matrix} \\
  \hline
  \begin{matrix}
  \minus\lambda_i  &        &   & \\
  \minus1   & \minus\lambda_i &     &\\
      & \ddots & \ddots &  \\
      &        & \minus1   & \minus\lambda_i \\
  \end{matrix} & 0
 \end{array}
 \right)
\quad  B_i= \left(
\begin{array}{c|c}
  0 & \begin{matrix}
    1 & &        & \\
      & 1 &  &     \\
      &        & \ddots &   \\
      &        &        & 1   \\
    \end{matrix} \\
  \hline
  \begin{matrix}
  \minus1  &        &   & \\
     & \minus1 &     &\\
      &  & \ddots &  \\
      &        &    & \minus1 \\
  \end{matrix} & 0
 \end{array}
 \right)
\end{equation}} \item Jordan block with eigenvalue $\infty$ {\scriptsize \begin{equation} \label{Eq:JordBlockInf}
A_i = \left(
\begin{array}{c|c}
  0 & \begin{matrix}
   1 & &        & \\
      &1 &  &     \\
      &        & \ddots &   \\
      &        &        & 1   \\
    \end{matrix} \\
  \hline
  \begin{matrix}
  \minus1  &        &   & \\
     & \minus1 &     &\\
      &  & \ddots &  \\
      &        &    & \minus1 \\
  \end{matrix} & 0
 \end{array}
 \right)
\quad B_i = \left(
\begin{array}{c|c}
  0 & \begin{matrix}
    0 & 1&        & \\
      & 0 & \ddots &     \\
      &        & \ddots & 1  \\
      &        &        & 0   \\
    \end{matrix} \\
  \hline
  \begin{matrix}
  0  &        &   & \\
  \minus1   & 0 &     &\\
      & \ddots & \ddots &  \\
      &        & \minus1   & 0 \\
  \end{matrix} & 0
 \end{array}
 \right)
 \end{equation} } \item   Kronecker block {\scriptsize \begin{equation} \label{Eq:KronBlock} A_i = \left(
\begin{array}{c|c}
  0 & \begin{matrix}
   1 & 0      &        &     \\
      & \ddots & \ddots &     \\
      &        & 1    &  0  \\
    \end{matrix} \\
  \hline
  \begin{matrix}
  \minus1  &        &    \\
  0   & \ddots &    \\
      & \ddots & \minus1 \\
      &        & 0  \\
  \end{matrix} & 0
 \end{array}
 \right) \quad  B_i= \left(
\begin{array}{c|c}
  0 & \begin{matrix}
    0 & 1      &        &     \\
      & \ddots & \ddots &     \\
      &        &   0    & 1  \\
    \end{matrix} \\
  \hline
  \begin{matrix}
  0  &        &    \\
  \minus1   & \ddots &    \\
      & \ddots & 0 \\
      &        & \minus1  \\
  \end{matrix} & 0
 \end{array}
 \right)
 \end{equation} }
 \end{itemize}

\end{theorem}

Each Kronecker block is a $(2k_i-1) \times (2k_i-1)$ block, where
$k_i \in \mathbb{N}$. If $k_i=1$, then the blocks are $1\times 1$
zero matrices \[A_i =
\begin{pmatrix}
0
\end{pmatrix}, \qquad B_i=
\begin{pmatrix}
0
\end{pmatrix}.\]  We call a decomposition of $V$ into a sum of subspaces corresponding to the Jordan and Kronecker blocks a \textbf{Jordan-Kronecker decomposition}:  \begin{equation} \label{Eq:JKDecomp} V= \bigoplus_{j=1}^N  V_{J_{\lambda_j, 2n_j}}\oplus  \bigoplus_{i=1}^q V_{K_i}.\end{equation} We call the one-parametric family of skew-symmetric forms $\mathcal{L} = \left\{A + \lambda B \, \, \bigr| \,\, \lambda \in \bar{\mathbb{C}} \right\}$ a \textbf{linear pencil}.

\subsubsection{Real Jordan--Kronecker theorem}

There exists a natural real analog of the Jordan–Kronecker theorem.

\begin{theorem} Any two skew-symmetric bilinear forms A and B on a real finite-dimensional vector space $V$ can be reduced simultaneously to block-diagonal form; besides, each block is either
a Kronecker block or a Jordan block with eigenvalue $\lambda \in \mathbb{R} \cup \left\{\infty \right\}$ or a real Jordan block with
complex eigenvalue $\lambda = \alpha + i \beta$:
\[A_i =\left(
\begin{array}{c|c}
  0 & \begin{matrix}
   \Lambda &E&        & \\
      & \Lambda & \ddots &     \\
      &        & \ddots & E  \\
      &        &        & \Lambda   \\
    \end{matrix} \\
  \hline
  \begin{matrix}
  \minus\Lambda  &        &   & \\
  \minus E   & \minus\Lambda &     &\\
      & \ddots & \ddots &  \\
      &        & \minus E   & \minus \Lambda \\
  \end{matrix} & 0
 \end{array}
 \right)
\quad  B_i= \left(
\begin{array}{c|c}
  0 & \begin{matrix}
    E & &        & \\
      & E &  &     \\
      &        & \ddots &   \\
      &        &        & E   \\
    \end{matrix} \\
  \hline
  \begin{matrix}
  \minus E  &        &   & \\
     & \minus E &     &\\
      &  & \ddots &  \\
      &        &    & \minus E  \\
  \end{matrix} & 0
 \end{array}
 \right) \]
Here $\Lambda$ and $E$ are the $2 \times 2$ matrices
\[ \Lambda =\left( \begin{matrix} \alpha & - \beta \\ \beta & \alpha \end{matrix} \right), \qquad 
 E = \left( \begin{matrix} 1 & 0 \\ 0 & 1 \end{matrix} \right).\]
\end{theorem}

In the real case in the Jordan-Kronecker decomposition we should "group together" subspaces corresponding to pairs of complex conjugate eigenvalues $\alpha_j \pm i \beta_j$:  \[  V= \bigoplus_{j=1}^{N_1}  V_{J_{\lambda_j, 2n_j}} \oplus  \bigoplus_{j=1}^{N_2}  V_{J_{\alpha_j \pm \beta_j, 4m_j}} \oplus  \bigoplus_{i=1}^q V_{K_i}.\]

\subsubsection{Characteristic polynomial}
The \textbf{rank} of a linear pencil $\mathcal{L} = \left\{ A + \lambda B\right\}$ is \[ \operatorname{rk} \mathcal{L} = \max_{\lambda \in\bar{\mathbb{C}}} \operatorname{rk} (A +\lambda B).\] A value $\lambda_0 \in \bar{\mathbb{C}}$ is \textbf{regular} if $\operatorname{rk} A_{\lambda_0} = \operatorname{rk} \mathcal{L}$. We also call $A_{\lambda_0}$ a regular form of the pencil $\mathcal{L}$.

\begin{definition} Let $\mathcal{L} = \left\{A+ \lambda B \right\}$ be a linear pencil with rank $\operatorname{rk} \mathcal{L} = r$. Take all $r \times r$ diagonal minors $\Delta_{i_1,\dots i_r}$ of the matrix\footnote{Note the ``minus" sign. We take the matrix $A - \lambda B$ in order to get $(\lambda - \lambda_0)^p$ as the characteristic polynomial for a $p \times p$ Jordan block with eigenvalue $\lambda_0$.} $A - \lambda B$. The \textbf{characteristic polynomial} $p_{\mathcal{L}}(\lambda)$ is the greatest common divisor of the Pfaffians of these minors: \[ p_{\mathcal{L}}(\lambda) = \operatorname{gcd} \left\{ \operatorname{Pf}\left(\Delta_{i_1,\dots i_r} \right) \, \, \bigr| \, \, 1 \leq i_1 < \dots < i_r \leq n \right\} .\]

\end{definition}

\begin{remark} The characteristic polynomial  $p_{\mathcal{L}}(\lambda)$ is defined up to multiplication by a constant. In this paper we consider pencils $A+\lambda B$ with finite eigenvalues (i.e. $B$ is regular). To avoid ambiguity we consider $p_{\mathcal{L}}(\lambda)$ to be a monic polynomial (i.e. its leading coefficient is equal to $1$): \[  p_{\mathcal{L}}(\lambda) = \lambda^N + \dots . \]
\end{remark}

If $B$ is nondegenerate, we can consider the recursion operator $P = B^{-1} A$. Then the characteristic polynomial is given by \[\operatorname{det}(P - \lambda E) = p_{\mathcal{L}}(\lambda)^2. \] In particular, for one Jordan $p \times p$ block with eignevalue $\lambda_0$  the characteristic polynomial $p_{\mathcal{L}}(\lambda)$ is $\left(\lambda - \lambda_0 \right)^p$. In the general case, we get the product of such polynomials for the Jordan blocks in a Jordan--Kronecker decomposition \eqref{Eq:JKDecomp}: \[p_{\mathcal{L}}(\lambda) = \prod_{j=1}^{N} \left(\lambda - \lambda_j\right)^{n_j} \] (for details, see e.g. \cite{{Gantmaher88}}).

\subsubsection{Core and mantle subspaces}

\begin{definition} Consider a pencil of skew-symmetric forms $\left\{ A_{\lambda} = A + \lambda B\right\}$.

\begin{itemize} 

\item The \textbf{core} subspace is \[ K = \sum_{\lambda - regular} \operatorname{Ker} A_{\lambda}. \]

\item The \textbf{mantle} subspace is direct sum of the core subspace and the Jordan blocks from a Jordan--Kronecker decomposition \eqref{Eq:JKDecomp}: \[ M = K \oplus  \bigoplus_{j=1}^N  V_{J_{\lambda_j, 2n_j}}. \]

\end{itemize}

\end{definition} 

There is a simple description of the core subspace. We call a basis $e_1, \dots, e_{k_i-1},  f_1, \dots, f_{k_i}$   of a Kronecker block \textbf{standard} if the linear pencil has the form \eqref{Eq:KronBlock}.

\begin{proposition} \label{Cor:CoreMantle} For any JK decomposition~\eqref{Eq:JKDecomp} the core subspace $K$ of $V$ is the direct sum of core subspaces of Kronecker subspaces $V_{K_i}$: \[ K = \bigoplus_{i=1}^q \left(K \cap V_{K_i}\right).\] If $e_1, \dots, e_{k_i-1},  f_1, \dots, f_{k_i}$ is a standard basis of $V_{K_i}$, then the core subspace of $V_{K_i}$ is \[K \cap V_{K_i} = \operatorname{span} \left(f_1, \dots, f_{k_i} \right).\] \end{proposition} 

Simply speaking, the core subspace $K$ is spanned by the subspaces corresponding to the down-right blocks of the Kronecker blocks. We would need the following simple statement.

\begin{proposition} \label{Prop:CommutLinear} Let $\mu_1, \dots, \mu_D \in \mathbb{C} \cup \left\{\infty\right\}, D \in \mathbb{N}$ be any distinct values,  $v_i \in \operatorname{Ker} (A + \mu_i B), i=1, \dots, D$ be any vectors. Consider the subspace \[ U = K + \operatorname{span}\left\{ v_1, \dots, v_D \right\}\] where $K$ is the core subspace.

\begin{enumerate}
    \item $U$ is isotropic w.r.t. all forms $A + \lambda B, \lambda \in \bar{\mathbb{C}}$.
    \item  If $v_j \not \in K$, then $-\mu_j$ is an eigenvalue (i.e $\mu_j$ is not a regular value) and the vector $v_j$ belongs to one of the corresponding Jordan blocks:  \[ v_j \in V_{J_{-\mu_j, 2m_j}}. \] 
    \item Therefore, \[ \dim U = \dim K + \left| \left\{j \,\, \bigr| \,\, v_j \not \in K \right\} \right|. \] 
\end{enumerate} 
\end{proposition}

\subsection{Poisson pencils}

Let $M$ be a real $\mathbb{C}^{\infty}$-smooth or a complex analytic manifold. 

\begin{itemize}

\item Two Poisson brackets  $\mathcal{A}$ and $\mathcal{B}$ on $M$ are  \textbf{compatible} if any their linear combination with constant coefficients $\alpha \mathcal{A} + \beta \mathcal{B}$ is also a Poisson bracket (in practice, it suffices to check that $\mathcal{A} + \mathcal{B}$ is a Poisson bracket). 

\item We call a pencil of compatible Poisson brackets $\mathcal{P} = \left\{ \mathcal{A} + \lambda \mathcal{B} \right\} $, where $\lambda \in \bar{\mathbb{C}} = \mathbb{C} \cup \left\{ \infty \right\}$, a \textbf{Poisson pencil}. We use the following notations:  \[\mathcal{A}_{\lambda} = \mathcal{A} + \lambda \mathcal{B}, \qquad \mathcal{A}_{\infty} = \mathcal{B}.\] 

\item Functions $f$ and $g$ \textbf{commute} (or are in \textbf{involution}) w.r.t. to a Poisson bracket $\mathcal{A}_{\lambda}$ if $\left\{ f, g\right\}_{\lambda} := \mathcal{A}_{\lambda}(df, dg) = 0$.

\item A \textbf{Casimir function} of a Poisson bracket $\mathcal{A}$ is a function $f$ commutes with all other functions w.r.t. this bracket. The set of Casimir function is denoted by $\mathcal{C} \left( \mathcal{A}\right)$.

\end{itemize}

\subsubsection{Core distribution}

In this section we discuss distributions on $M$ equipped with a Poisson pencil  $\mathcal{P}$. For the terminology and more details about singular distributions and their integrability see \cite{DufourZung05}. 

\begin{itemize}

\item  A \textbf{singular distribution} on a manifold $M$ is the assignment to each point $x$ of M a vector subspace $D_x$ of the tangent space $T_xM$. The dimension of $D_x$ may depend on $x$.

\item A singular distribution $D$ on a smooth manifold is called \textbf{smooth} if for any point $x$ of $M$ and any vector $X_0 \in D_x$, there is a smooth vector field $X$ defined in
a neighborhood $U_x$ of $x$ which is tangent to the distribution $D$. 

\item We say that a smooth singular distribution $D$ on a smooth manifold $M$ is 

\begin{itemize}
\item a \textbf{regular}\footnote{In other words, a regular distribution is a a subbundle of $TM$.} distribution if $\dim D_x$ does not depend on $x$;

\item an \textbf{integrable} distribution if every point $x \in M$ is contained in an a connected immersed submanifold $N_x \subset M$ such that $T_x N_x = D_x$.

\end{itemize}

\item For any distribution $\Delta \subset TM$ we can also consider its dual distribution $\Delta^0 \subset T^*M$.  A distribution $\Delta^0$ in the cotangent bundle $T^*M$ will be called \textbf{integrable} (\textbf{involutive}, etc) if $\Delta$ is integrable (respectively, involutive, etc).

\item Let $\mathcal{P}$ be a Poisson pencil on $M$. A distribution $\Delta^0 \subset T^*M$ is \textbf{bi-isotropic} if each subspace $\Delta^0_x$ is a bi-isotropic subspace of $\left(T^*_x M, \mathcal{P}(x) \right)$

\end{itemize}

\begin{definition} \label{D:CoreMantleDist} Let $\mathcal{P} = \left\{ \mathcal{A}_{\lambda} = \mathcal{A} + \lambda \mathcal{B}\right\}$ be a Poisson pencil on $M$. There are two natural singular distributions in $T^*M$: 

\begin{enumerate}

\item The core subspace in each cotangent space $T^*_x M$ defines a  the \textbf{core distribution} $\mathcal{K}$ in $T^*M$. In other words, at each point $x \in M$
 \begin{equation} \label{Eq:CoreDist} \mathcal{K}_x =  \bigoplus_{\lambda - \text{regular for $\mathcal{P}(x)$}} \operatorname{Ker}\mathcal{A}_\lambda (x),  \end{equation} 

\item Similarly, the mantle subspace in each cotangent space $T^*_x M$ defines  the \textbf{mantle distribution} $\mathcal{M}$.
 
\end{enumerate}

 \end{definition}

We use the following simple statement about the core distribution (see e.g.\footnote{The first statement that in \cite[Proposition 5.3]{Kozlov23JKRealization} that $\mathcal{K}$ is an integrable singular smooth distribution is wrong. An intersection of integrable regular distributions a priori may not be smooth (although it will be integrable if it is smooth). Since we are not interested in the singularities and only consider the case  only interested in the case $\operatorname{rk} \mathcal{P}(x) = \operatorname{const}$ and $\operatorname{deg} p_{\mathcal{P} (x)} = \operatorname{const}$, it doesn't change anything.} \cite[Proposition 5.3]{Kozlov23JKRealization}).

\begin{proposition} \label{Prop:CoreDistDim}  Let $\mathcal{P} = \left\{\mathcal{A} + \lambda \mathcal{B} \right\}$ be a Poisson pencil on a manifold $M$ and $p_{\mathcal{P}(x)}$ be its characteristic polynomial at $ x\in M$.

\begin{enumerate} 

\item For any point $x \in M$ we have \[\operatorname{dim} \mathcal{K}_x = \operatorname{dim} M - \frac{1}{2} \operatorname{rk} \mathcal{P}(x) - \operatorname{deg} p_{\mathcal{P}(x)}. \]

\item If $\operatorname{rk} \mathcal{P}(x) = \operatorname{const}$ and $\operatorname{deg} p_{\mathcal{P} (x)} = \operatorname{const}$ in a neighbourhood $O_{x_0} \subset (M, \mathcal{P})$, then $\mathcal{K}$ is an integrable regular distribution in $O_{x_0}$.
\end{enumerate}

\end{proposition}

In practice we can generate the core distribution by taking a sufficient number of (local) Casimir functions. The next statement easily follows from the Jordan--Kronecker theorem.

\begin{proposition} \label{Prop:CasimirAndCore} Let $\mathcal{P} = \left\{\mathcal{A} + \lambda \mathcal{B} \right\}$ be a Poisson pencil on a manifold $M$ and $p_{\mathcal{P}(x)}$ be its characteristic polynomial at $ x\in M$.

\begin{enumerate}

\item For any point $x \in M$, any bracket $\mathcal{A}_{\lambda}, \lambda \in \bar{\mathbb{C}}$ such that $\mathcal{A}_{\lambda}(x)$ is regular in $\mathcal{P}(x)$ and any Casimir function $f \in \mathcal{C} \left(\mathcal{A}_{\lambda} \right)$ we have

\[df (x) \in \mathcal{K}_x. \]

\item Consider a Jordan--Kronecker decomposition of $\mathcal{P}(x)$ for a point $x \in M$. Assume that the biggest Kronecker block has size $(2D - 1) \times (2D -1), D > 0$. Then for any values $\lambda_1, \dots, \lambda_{D} \in \bar{\mathbb{C}}$ that are regular for the pencil $\mathcal{P}(x)$ we have \[ \mathcal{K}_x = \bigoplus_{j=1}^{D} \operatorname{Ker} \mathcal{A}_{\lambda_j}(x).\] 

\end{enumerate}

\end{proposition}

By the well-known Darboux--Weinstein Theorem, if a Poisson bracket $\mathcal{A}$ has constant rank $\operatorname{rk} \mathcal{A} = \operatorname{const}$, then  $\operatorname{Ker}\mathcal{A}(x)$ is spanned by differentials of local Casimir functions. Thus, the core distribution $\mathcal{K}$ is also locally spanned by differentials of local Casimir functions at a generic point. Formally, we have the following statement.

\begin{corollary} \label{Cor:LocSpan} Let $\mathcal{P}$ is a Poisson pencil with  on $M$. If $\operatorname{rk} \mathcal{P} = \operatorname{const}$ on $M$, then in a sufficiently small neighborhood $U$ of any point $x_0$ there exist Casimir functions $f_{j, 1}, \dots, f_{j, m_j} \in \mathcal{C}\left(\mathcal{A}_{\mu_j}\right), j = 1, \dots, D$ such that 

\begin{enumerate} 

\item  $\mathcal{A}_{\mu_j}(x), j=1,\dots, D$ are regular in the linear pencil $\mathcal{P}(x)$ for any $x \in U$;  

\item the core distribution $\mathcal{K}$ is locally spanned by  the differentials of  Casimir functions:

\[ \mathcal{K}_x =  \operatorname{span} \left\{  df_{1, 1}(x), \dots, df_{D, m_D}(x) \right\}, \qquad \forall x \in U. \]

\end{enumerate}

\end{corollary}

\subsection{Local coordinates for core and mantle}

We need the results of this section only in the real case, when a characteristic polynomial $p_{\mathcal{P}}$ of a Poisson pencil has  complex conjugate eigenvalues $\alpha \pm i \beta$. The following two useful Theorems~\ref{T:BiPoissRedCoreMantle} and \ref{T:TrivKronFact} often allows us reduce some problems about Poisson pencils $\mathcal{P}$ to the case when 

\begin{itemize}

\item there is only one eigenvalue $\lambda_0$ (or a pair of complex conjugate eigenvalues $\alpha_0 \pm i \beta_0$ in the real case),

\item and all Kronecker blocks of $\mathcal{P}$ are trivial $1\times 1$ blocks.

\end{itemize} The next theorem is proved in \cite[Theorem 5.9]{Kozlov23JKRealization}.

\begin{theorem}  \label{T:BiPoissRedCoreMantle}  Let $\mathcal{P} = \left\{ \mathcal{A}_{\lambda} = \mathcal{A} + \lambda \mathcal{B}\right\}$ be a Poisson pencil on $M$ and $p_{\mathcal{P}(x)}$ be its characteristic polynomial at $x\in M$. Assume that \[\deg p_{\mathcal{P}(x)} = \operatorname{const}, \qquad \operatorname{rk} \mathcal{P}(x) = \operatorname{const}\] on $M$.  Denote \[ n_J = \dim p_{\mathcal{P}(x_0)}, \qquad m  = \operatorname{rk} \mathcal{P}(x_0) - 2n_J, \qquad r = \operatorname{corank} \mathcal{P}.\]  
Then for any point $x \in M$ there exist local coordinates $x_1,\dots, x_{m}, s_1, \dots, s_{2n_J}, y_1, \dots, y_{m+r}$ such that the core and mantle distribution are \begin{equation} \label{Eq:CoreMantleDistLoc} \mathcal{K} = \operatorname{span}\left\{dy_1, \dots, dy_{m+r} \right\}, \qquad \mathcal{M} = \operatorname{span}\left\{ds_1, \dots, ds_{2n_J}, dy_1, \dots, dy_{m+r} \right\}\end{equation} and the pencil has the form \[\mathcal{A}_{\lambda} = \sum_{i=1}^{m} \frac{\partial}{\partial x_i} \wedge v_{\lambda, i} + \sum_{1 \leq i < j \leq 2 n_J} c_{\lambda, ij}(s, y)  \frac{\partial}{\partial s_i} \wedge \frac{\partial}{\partial s_j} \] for some vector fields $v_{\lambda, i} = v_{\lambda, i}(x, s, y)$ and some functions $c_{\lambda, ij}(s, y)$.  \end{theorem}

Simply speaking, in the coordinatex $(x, s, y)$ from Theorem~\ref{T:BiPoissRedCoreMantle}  the matrices of the Poisson brackets take the form \begin{equation} \label{Eq:CoreMantleMatr} \mathcal{A}_{\lambda} = \left( \begin{matrix} * & * & * \\ * & C_{\lambda}(s, y) & 0 \\ * & 0 & 0\end{matrix} \right), \end{equation}  where $*$ are some matrices. 

\subsubsection{Factorization theorem}

Consider the coordinates $(s,y)$ from Theorem~\ref{T:BiPoissRedCoreMantle} and the corresponding pencil \[\mathcal{A}_{\lambda}' =  \left( \begin{matrix}   C_{\lambda}(s, y) & 0 \\  0 & 0\end{matrix} \right).\] We can ``group'' the coordinates $s$ by eigenvalues. Formally, we have the following statement.

\begin{theorem}  \label{T:TrivKronFact} Let $\mathcal{P} = \left\{ \mathcal{A}_{\lambda} = \mathcal{A} + \lambda \mathcal{B}\right\}$ be a Poisson pencil on real smooth manifold $M$ and $p_{\mathcal{P}(x)}$ be its characteristic polynomial at $x\in M$. Assume the following:

\begin{enumerate}
    \item For all $x \in M$ we have \[\deg p_{\mathcal{P}(x)} =  \operatorname{rk} \mathcal{P}(x) = \operatorname{const}\] on $M$. In other words, a Jordan--Kronecker decomposition of a pencil $\mathcal{P}(x)$ consists of Jordan blocks and $r = \dim M - \operatorname{rk}\mathcal{P}$ trivial $1 \times 1$ Kronecker blocks. 

\item At a point $p \in M$ the characteristic polynomial $p_{\mathcal{P}(x)}$  has $k$ real (distinct) eigenvalues $\lambda_1, \dots, \lambda_k$  with multiplicities $m_1, \dots, m_k$ respectively and $s$ pairs of complex (non-real) conjugate eigenvalues $\mu_1, \bar{\mu}_1,\dots, \mu_s, \bar{\mu}_s$ with  multiplicities $l_1, \dots, l_s$.
    
\end{enumerate}  Then in a neighbourhood of $p \in M$ there exists a local coordinate
system
\begin{gather*} x_1 = \left(x_1^1, \dots, x_1^{2m_1}\right), \qquad \dots, \qquad x_k = \left(x_k^1, \dots, x_k^{2m_k}\right), \\ u_1 = \left(u_1^1, \dots, u_1^{4l_1}\right), \quad \dots, \quad u_s = \left(u_s^1, \dots, x_s^{4l_s}\right), \quad z = (z_1, \dots, z_r), \end{gather*} such that the matrices of Poisson brackets have the form \begin{equation} \label{Eq:FormOneEigen} \mathcal{A}_{\lambda} = \left( \begin{matrix} C^1_{\lambda}(x_1, z) & & & & & & & \\ & \ddots & & & & & \\ & & C^k_{\lambda}(x_k, z) & & & & & \\ & & & D^1_{\lambda}(u_1, z) & & & & \\ & & & & \ddots & & & \\ & & &  & & & D^s_{\lambda}(u_s, z)  & \\ & & &  & & &   & 0_r \end{matrix} \right). \end{equation} Moreover,  at the point $p \in M$ each characteristic polynomial of the pencils $\left\{ C^t_{\lambda}(x_t, z) \right\}$ has a single real eigenvalue. And each characteristic polynomial of the pencils $\left\{ D^t_{\lambda}(u_t, z) \right\}$ has a single pair of complex eigenvalues at $p \in M$. \end{theorem}

\begin{remark}
In the complex case we have a natural analog of  Theorem~\ref{T:TrivKronFact}. Simply speaking, we do not consider complex conjugate eigenvalues and get similar coordinates $x_1, \dots, x_k, z$.  \end{remark}

\begin{proof}[Proof of Theorem~\ref{T:TrivKronFact}] Since all Kronecker blocks are $1 \times 1$ all  regular forms $\mathcal{A}_{\lambda}$ have common (local) Casimir functions $z_1, \dots, z_r$. They also have the same symplectic leaves $(S_z, \omega_{\lambda, z})$, i.e. level sets of Casimir functions: \[ S_z = \left\{ z_1 = \operatorname{const}, \dots, z_r = \operatorname{const}\right\}. \] On each symplectic leaf $S_z$ the pencil $\mathcal{P}$ defines a  \textbf{nondegenerate}\footnote{A pair of nondegerenerate Poisson brackets $\mathcal{A}$ and $\mathcal{B}$ is compatible iff the recursion operator $P =\mathcal{A}\mathcal{B}^{-1}$ is a Nijenhuis operator, i.e. $N_{P} = 0$. Compatible nondegenerate Poisson brackets are the same as compatible symplectic forms $\mathcal{A}^{-1}$ and $\mathcal{B}^{-1}$.} Poisson pencil $\mathcal{P}^z$. We can easily "split" the nondegenerate pencils $\mathcal{P}^z$ using \cite[Lemma 2]{turiel}. Alternatively, one can use  the splitting theorem for Nijenhuis operators (see \cite[Theorem 3.1]{BolsinovNijenhuis}) We get coordinates $x_1, \dots, x_k, u_1, \dots u_s$ such that the matrices of the pencils $\mathcal{P}^z$ are block-diagonal: \[ \mathcal{P}^z =  \left( \begin{matrix} C^1_{\lambda}(x_1, z) & &  \\ & \ddots &  \\ & & D^s_{\lambda}(u_s, z) \end{matrix} \right). \] Since $z_i$ are Casimir function, the pencil $\mathcal{P}$ takes the form \eqref{Eq:FormOneEigen}. Theorem~\ref{T:TrivKronFact} is proved. \end{proof}

\subsection{Eigenvalues of Poisson pencils}

Lemmas~\ref{L:EigenLemma} and \ref{L:EigenLemmaRealConj} are the key technical results underlying the proof of Theorem~\ref{Th:CasEigenComm}. Actually, the rest of the proof is simple Linear algebras.  Although we were not able to find the statement of Lemma~\ref{L:EigenLemma} in
the literature, it is well-known to the experts in the field. For nondegenerate pencils
Lemma~\ref{L:EigenLemma} follows from a similar statement about eigenvalues of Nijenhuis operators
(see \cite[Proposition 2.3]{BolsinovNijenhuis}). We slightly generalize the statement of \cite[Lemma 9.8]{Kozlov23JKRealization} by providing a more refined condition on the eigenvalue,  the proof remains the same.

\begin{definition} Let $\mathcal{P} = \left\{\mathcal{A} + \lambda \mathcal{B} \right\}$ be a Poisson pencil on $M$. We say that its eigenvalue $\lambda(x)$ is an \textbf{isolated} eigenvalue if there is a neighborhood $U \subset M \times \mathbb{C}$ of the graph \begin{equation} \label{Eq:Graph} \left\{(x, \lambda(x)) \,\, \bigr| \,\, x \in M\right\} \subset M \times \mathbb{C} \end{equation} that has no other eigenvalues of $\mathcal{P}$.  \end{definition}

\begin{lemma} \label{L:EigenLemma} Let $\mathcal{P} = \left\{\mathcal{A} + \lambda \mathcal{B} \right\}$ be a Poisson pencil on $M$ and $\lambda(x)$ be its (finite) eigenvalue on $M$. In the real case $\lambda(x)$ is real $C^{\infty}$-smooth and in the complex case it is complex-analytic. Assume that the following two conditions hold:

\begin{enumerate}

\item  $\operatorname{rk} \mathcal{P}(x) = \operatorname{const}$. 

\item $\lambda(x)$ is an isolated eigenvalue. 

\end{enumerate}

Then for any point $x \in M$ we have\begin{equation} \label{Eq:DLambdaMain} (\mathcal{A} - \lambda(x)\mathcal{B})d\lambda(x) = 0. \end{equation}
\end{lemma}

Formula~\eqref{Eq:DLambdaMain} can also be rewritten as \[ d \lambda(x) \in \operatorname{Ker} \mathcal{A}_{-\lambda(x)}.\] Note that we do not consider pairs of complex conjugate eigenvalues in Lemma~\ref{L:EigenLemma} (cmp. Lemma~\ref{L:EigenLemmaRealConj})

\begin{proof}[Proof of Lemma~\ref{L:EigenLemma}] Let $x_0 \in M$, denote $\lambda(x_0) = \lambda_0$. Notice that $\lambda_0$ is an eigenvalue at $ x \in M$ iff \[ \operatorname{rk} \left( \mathcal{A} - \lambda_0 \mathcal{B}\right) \bigr|_x < \operatorname{rk} \mathcal{P}.\]
Let $S$ be the symplectic leaf of $\mathcal{A} - \lambda_0 \mathcal{B}$ through $x_0$. Then \[\operatorname{dim} S = \operatorname{dim} \operatorname{Im} \left( \mathcal{A} - \lambda_0 \mathcal{B}\right) \bigr|_x = \operatorname{rk}\left( \mathcal{A} - \lambda_0 \mathcal{B}\right) \bigr|_x \] 
for any $x \in S$. Thus $\lambda_0$ is an eigenvalue on $S$. Since $\lambda(x)$ is an isolated eigenvalue,  $\lambda(x) = \lambda_0$ on $S$. We get that $d \lambda(x) = 0$ on $T_x S =\operatorname{Im} \left( \mathcal{A} - \lambda_0 \mathcal{B}\right) \bigr|_x$ for $x \in S$. Thus, \[d \lambda_i(x_0) \in \left(\operatorname{Im} \left( \mathcal{A} - \lambda_0 \mathcal{B}\right) \bigr|_{x_0}\right)^{0} =  \operatorname{Ker} \left( \mathcal{A} - \lambda_0 \mathcal{B}\right) \bigr|_{x_0},\]
which proves \eqref{Eq:DLambdaMain} and Lemma~\ref{L:EigenLemma}.
\end{proof}

In practice it may be convenient to check the second condition of Lemma~\ref{L:EigenLemma} using the following simple statement.

\begin{proposition} \label{Prop:IsolEigen} Let $\mathcal{P} = \left\{\mathcal{A} + \lambda \mathcal{B} \right\}$ be a Poisson pencil on $M$, $p_{\mathcal{P}}$ be its characteristic polynomial and $\lambda(x)$ be its (finite) smooth eigenvalue on $M$. If $\deg p_{\mathcal{P}}(\lambda)= \operatorname{const}$ on $M$ and the multiplicity of $\lambda(x)$ is constant on $M$, then the eigenvalue $\lambda(x)$ is an isolated eigenvalue.  \end{proposition}

\begin{proof} It is well-known that the n roots of a polynomial of degree n depend continuously on the coefficients (that can be proved using Rouch\'{e}’s Theorem, see e.g. \cite{Alexanderian13}). Assume that in a neighborhood $U \subset M$ the eigenvalues are $\lambda_1(x), \lambda_2(x), \dots, \lambda_N(x)$. Since the multiplicity of $\lambda(x)$ is constant, either $\lambda_j(x) \not= \lambda(x)$ or $\lambda_j(x) = \lambda(x)$ for all $x \in M$. Thus, there are no other eigenvalues in a sufficiently small neighborhood of the graph \eqref{Eq:Graph}. Proposition~\ref{Prop:IsolEigen} is proved. \end{proof}

\subsubsection{Complex conjugate eigenvalues}

In the real case we will  will also need the following analog of Lemma~\ref{L:EigenLemma} for complex conjugate eigenvalues. 

\begin{lemma} \label{L:EigenLemmaRealConj} Let $\mathcal{P} = \left\{\mathcal{A} + \lambda \mathcal{B} \right\}$ be a Poisson pencil on real manifold $M$ and $\lambda(x) = \alpha(x) + i \beta(x)$ be its complex eigenvalue on $M$. Then almost everywhere on $M$  we have \begin{equation} \label{Eq:DLambdaMainRealConj} d\lambda(x) = d\alpha(x) + i \cdot d \beta(x) \in \operatorname{Ker}_{-\lambda(x)}^{\mathbb{C}} + \mathcal{K}^{\mathbb{C}}.  \end{equation} 
\end{lemma}

Here at each point $x \in M$ we complexify the cotangent space $T^* M$ and extend $\mathcal{A}(x)$ and $\mathcal{B}(x)$ to the skew-symmetric forms $\mathcal{A}^{\mathbb{C}} (x)$ and $\mathcal{B}^{\mathbb{C}} (x)$ on $(T^*M)^{\mathbb{C}}$. Then $\mathcal{K}^{\mathbb{C}}$ is the complexification of the core distribution $\mathcal{K}$ and  \[\operatorname{Ker}_{-\lambda(x)}^{\mathbb{C}} = \operatorname{Ker}\left(\mathcal{A}^{\mathbb{C}}(x) - \lambda(x) \mathcal{B}^{\mathbb{C}}(x)\right). \]

\begin{proof}[Proof of Lemma~\ref{L:EigenLemmaRealConj}]
Immediately follows from Theorems~\ref{T:BiPoissRedCoreMantle} and \ref{T:TrivKronFact} about local structure of compatible Poisson brackets. Lemma~\ref{L:EigenLemmaRealConj} is proved. 
\end{proof}

\section{Main result}

\begin{theorem} \label{Th:CasEigenComm} Let $\mathcal{P} = \left\{ \mathcal{A} + \lambda \mathcal{B} \right\}$ be a Poisson pencil on a manifold $M$ and $p_{\mathcal{P}}(\lambda)$ be its characteristic polynomial. Assume that $\deg p_{\mathcal{P}}(\lambda) = \operatorname{const}$ on $M$ and  \begin{equation} \label{Eq:CharOnM} p_{\mathcal{P}} (\lambda) = p_0 + p_1 \lambda + \dots + p_{N-1} \lambda^{N-1} + \lambda^N.\end{equation} Let $\mathcal{A}_{\mu_1}, \dots \mathcal{A}_{\mu_d} \in \mathcal{P}, \mu_j \in \bar{\mathbb{C}}$, be any $d\geq 0$ Poisson brackets that are regular on an open dense subset of $M$. Then for any Casimir functions $f_{j,k} \in \mathcal{C}\left( \mathcal{A}_{\mu_j}\right), j=1, \dots, d, k = 1,\dots, m_j,$ the set of functions \[ \mathcal{F} = \left\{ f_{1,1}, \dots, f_{d, m_d} \right\} \cup \left\{p_0, \dots, p_{N-1} \right\}\] are in involution w.r.t. all brackets from the pencil $\mathcal{A}_{\lambda} = \mathcal{A} + \lambda \mathcal{B}, \lambda \in \bar{\mathbb{C}}$.
\end{theorem}

\begin{remark} In Theorem~\ref{Th:CasEigenComm}:

\begin{enumerate}

\item We forbid non-finite eigenvalues $\lambda = \infty$ (formally, it is not a root of  \eqref{Eq:CharOnM}). In other words, we assume that $\operatorname{rk}\mathcal{B}(x) = \operatorname{rk}\mathcal{P}(x)$ for all $x \in M$.

\item We formally allow the pure Kronecker case, i.e. $N = 0$. Then, there is no characteristic polynomial $p_{\mathcal{P}}(\lambda)$ and we get a well-known statement that the Casimir functions of Poisson brackets $f_{j,k}$ are in bi-involution.

\item We also formally allow the pure Jordan case, i.e. $d = 0$. We get that the coefficients $p_j$ are in bi-involution.

\end{enumerate}

\end{remark}

It will be convenient to replace the coefficients $p_j$ with the roots $\lambda_j$ of the polynomial. Recall that the $n$ roots of a polynomial of degree $n$ depend continuously on the coefficients. 

\begin{proposition}\label{Prop:EigenSpan} Let $p(x, \lambda) = p_0(x) + p_1(x) \lambda + \dots + p_{N-1}(x) \lambda^{N-1} + \lambda^N$ be a polynomial in $\lambda$ on a manifold $M$. All $p_j(x)$ are smooth functions on $M$. Assume that in a neighborhood $Ox_0$ of a point $x_0 \in M$ the roots $\lambda_1(x), \dots, \lambda_N(x)$ have constant multiplicities\footnote{In other words, if two roots are equal at one point $\lambda_j(x_0) = \lambda_j(x_0)$, then they are equal in a neighborhood of that point $Ox_0$.}. Then the roots $\lambda_j(x)$ are analytic functions of the coefficients $p_0(x), \dots, p_{N-1}(x)$. Moreover, in the complex case  \begin{equation} \label{Eq:EqualSpan}  \operatorname{span} \left\{ dp_0, \dots, dp_{N-1}  \right\}  =  \operatorname{span} \left\{ d\lambda_1, \dots, d\lambda_N \right\} \end{equation} at each point of $Ox_0$. In the real case, assume that the first $2t$ roots are complex conjugate pairs $\alpha_j(x) \pm \beta_j(x)$ and the other roots $\lambda_j(x), j > 2t$, are real. Then \begin{equation} \label{Eq:RealSpan}  \operatorname{span} \left\{ dp_0, \dots, dp_{N-1} \right\} =  \operatorname{span} \left\{ d\alpha_1, d \beta_1 \dots, d\alpha_t,d \beta_t, d\lambda_{2t+1} \dots, d\lambda_{N} \right\}. \end{equation} \end{proposition}

\begin{proof}[Proof of Proposition~\ref{Prop:EigenSpan} ] The analyticity of $\lambda_j(x)$ as functions of the coefficients is proved\footnote{Another way to prove smoothness of the roots $\lambda_j$, without using \cite{Brillinger66}, is as follows. If a root $\lambda_j$ has multiplicity $1$, then it is smooth by the Implicit Function Theorem. If locally the root $\lambda_j$ has constant multiplicity $m$, then we can reduce the problem to the previous case by  taking the derivative of the polynomial $\frac{\partial}{\partial \lambda^k } p(x, \lambda)$ until we get a root with degree $1$.} in \cite{Brillinger66}). Since $p_0, \dots, p_{N-1}$ and $\lambda_1, \dots, \lambda_N$ are analytic functions of each other\footnote{On one hand, $\lambda_j$ are analytic functions of $p_j$ by  \cite{Brillinger66}. On the other hand, we have Vieta's formulas.} we get \eqref{Eq:EqualSpan} (or \eqref{Eq:RealSpan} in the real case).  Proposition~\ref{Prop:EigenSpan} is proved. \end{proof}

\begin{remark} Proposition~\ref{Prop:EigenSpan} shows that it is often a matter of preferences, whether to consider eigenvalues $\lambda_j(x)$ or coefficients $p_j$ of a characteristic polynomial $p_{\mathcal{P}}$. In practice, the coefficients $p_j$ may possess the following advantages:

\begin{itemize}

    \item \textit{Reduced Singularity}: The coefficients of the characteristic polynomial may have fewer singularities compared to the eigenvalues.

    \item \textit{Real-Valued Behavior}: For real-coefficient polynomials, the coefficients are guaranteed to be real numbers, while some eigenvalues might have complex conjugates. 
    
\end{itemize}

 \end{remark}

\begin{proof}[Proof of Theorem~\ref{Th:CasEigenComm}] Using Proposition~\ref{Prop:EigenSpan} in a neighborhood of a generic point of $M$ we can replace coefficients $p_0, \dots, p_{N-1}$ with locally analytic (complex-valued) roots $\lambda_1(x), \dots, \lambda_N(x)$.  In a neighborhood of a generic point $\operatorname{rk} \mathcal{P}(x) = const$ (the constant here may be different for different points). 

\begin{itemize}

    \item In the real case when $\lambda(x)$ is real or in the complex case we can use Lemma~\ref{L:EigenLemma}. Almost everywhere on $M$  we have \begin{equation} \label{Eq:DLambdaKer} d \lambda_j(x) \in \operatorname{Ker} \mathcal{A}_{-\lambda_j(x)}. \end{equation}

\item In the real case for non-real eigenvalues\footnote{In the real analytic case, when there are non-real eigenvalues $\lambda_j(x) = \alpha_j(x) + i \beta_j(x)$, \eqref{Eq:DLambdaKer} holds if we complexify everything in local coordinates. } $\lambda_j(x) = \alpha_j(x) + i \beta_j(x)$ we can use Lemma~\ref{L:EigenLemmaRealConj}. For a generic $x\in M$ in the complexified cotangent space $(T^*_xM)^{\mathbb{C}}$ we have \[d\lambda_j(x) \in \operatorname{Ker}_{-\lambda_j(x)}^{\mathbb{C}} + \mathcal{K}^{\mathbb{C}}.\]
    
\end{itemize}

 The rest is simple Linear Algebra. Since $df_{j,k} \in \operatorname{Ker} \mathcal{A}_{\mu_j}$ and the brackets $\mathcal{A}_{\mu_j}$ are regular, we have $df_{j,k} \in \mathcal{K}$ almost everywhere. By 
Proposition~\ref{Prop:CommutLinear} almost everywhere on $M$ we have \begin{equation} \label{Eq:CommF} \left\{ f, g\right\}_{\lambda} = 0, \qquad \forall f, g \in \mathcal{F}, \quad \forall \lambda \in \bar{\mathbb{C}}. \end{equation} By continuity arguments \eqref{Eq:CommF} holds on all $M$. Theorem~\ref{Th:CasEigenComm}  is proved. \end{proof}

\subsection{Extended core distribution}

Globally, instead of local Casimir functions, we can use the core distribution (compare with Proposition~\ref{Prop:CasimirAndCore}). 

\begin{definition} Assume that the characteristic polynomial $p_{\mathcal{P}}(\lambda)$ of a Poisson pencil $\mathcal{P}$ on $M$ is given by \eqref{Eq:CharOnM}. Let \begin{equation} \label{Eq:ExtCore} \hat{\mathcal{K}} = \mathcal{K} \plus  \operatorname{span} \left\{ dp_0, \dots, dp_{N-1} \right\}, \end{equation} where $\mathcal{K}$ is the core distribution. We call the singular distribution $\hat{\mathcal{K}}$ the \textbf{extended core distribution} of $\mathcal{P}$. \end{definition}

The next statement is a global analogue of Theorem~\ref{Th:CasEigenComm}.

\begin{lemma} \label{L:ExtCoreLemma} Let $\mathcal{P}$ be a Poisson pencil on a manifold $M$ such that \begin{equation} \label{Eq:CondRkDeg} \operatorname{rk} \mathcal{P}(x) = \operatorname{const}, \qquad \deg p_{\mathcal{P}}(\lambda) = \operatorname{const} \end{equation} on $M$. Then the following holds: 

\begin{enumerate}

\item The extended core distribution $\hat{\mathcal{K}}$ is a singular bi-isotropic distribution in $T^*M$.

\item Moreover, if $\dim \hat{\mathcal{K}} = \operatorname{const}$ on $M$, then $\hat{\mathcal{K}}$ is a regular integrable distribution.

\end{enumerate}

\end{lemma}

\begin{proof}[Proof of Lemma~\ref{L:ExtCoreLemma} ]

\begin{enumerate}

\item By Corollary~\ref{Cor:LocSpan} $\mathcal{K}$ is "spanned by local Casimir functions". Thus, $\hat{\mathcal{K}}$ is bi-isotropic by Theorem~\ref{Th:CasEigenComm}.

\item The core distribution $\mathcal{K}$ is regular by Proposition~\ref{Prop:CoreDistDim}. Thus, locally there exists functions $f_1, \dots, f_t$ such that \[\mathcal{K} =  \operatorname{span} \left\{ df_1,\dots, df_t \right\}.\] Thus, the extended core distribution $\hat{\mathcal{K}}$ locally has the form \begin{equation} \label{Eq:ExCoreSpan} \hat{\mathcal{K}} =  \operatorname{span} \left\{  df_1, \dots, df_t, dp_0, \dots, dp_{N-1} \right\}. \end{equation} Since the dimension of $\hat{\mathcal{K}}$ is constant, it is smooth and integrable. 
\end{enumerate}

Lemma~\ref{L:ExtCoreLemma} is proved.  \end{proof}

In order to apply Lemma~\ref{L:ExtCoreLemma} we need to know $\operatorname{dim} \hat{\mathcal{K}}$. We can find it using the following statement.

\begin{proposition} \label{Prop:DimExtCore} Let $\mathcal{P}$ be an analytic Poisson pencil on $M$ such that \eqref{Eq:CondRkDeg} holds and $\lambda_1(x), \dots, \lambda_N(x)$ be the roots of the characteristic polynomial $p_{\mathcal{P}}(x)$. Assume that in a neighborhood of $x_0 \in M$ the roots $\lambda_j(x)$ have constant multiplicities and there are $D$ distinct roots $\lambda_1(x), \dots, \lambda_D(x)$. Then \begin{equation} \label{Eq:DimExtCore} \dim \hat{\mathcal{K}}_{x_0} = \dim \mathcal{K}_{x_0} + \left| \left\{ j  \, \, \bigr| \, \, d \lambda_j(x_0) \not \in \mathcal{K}_{x_0}^{\mathbb{C}}, \quad j=1,\dots, D \right\} \right|, \end{equation} where $\mathcal{K}^{\mathbb{C}}$ is the complexification of $\mathcal{K}$. \end{proposition}

\begin{proof}[Proof of Proposition~\ref{Prop:DimExtCore}] Obviously, we can complexify everything at the point $x_0$ and consider vectors in $\left(T^*_{x_0}M\right)^{\mathbb{C}}$. By Proposition~\ref{Prop:EigenSpan} we have \[ \dim \hat{\mathcal{K}}_{x_0} = \dim \hat{\mathcal{K}}^{\mathbb{C}}_{x_0} = \dim \left( \mathcal{K}^{\mathbb{C}} \plus  \operatorname{span} \left\{ d\lambda_1, \dots, d\lambda_D \right\} \right).\] By Propositions~\ref{L:EigenLemma} (and  Proposition~\ref{L:EigenLemmaRealConj} in the real case) if  $\lambda_j(x_0) \not \in \hat{\mathcal{K}}^{\mathbb{C}}_{x_0}$, then $ d\lambda_j(x_0) \in \operatorname{Ker} \mathcal{A}^{\mathbb{C}}_{-\lambda_j(x_0)}$. Thus, \eqref{Eq:DimExtCore} follows from Proposition~\ref{Prop:CommutLinear}. Proposition~\ref{Prop:DimExtCore} is proved.  \end{proof}

Obviously, in a neighborhood of a generic point \eqref{Eq:CondRkDeg} is satisfied and  $\operatorname{dim}\hat{\mathcal{K}}$, given by \eqref{Eq:ExCoreSpan}, is constant. Thus, we get the following.  

\begin{corollary}  Let $\mathcal{P}$ be a Poisson pencil on $M$. In a neighborhood of a generic point the extended core distribution $\hat{\mathcal{K}}$ is a  regular integrable bi-isotropic distribution.
\end{corollary}

\section{Completeness criterion} \label{S:CompCrit}

\begin{definition} Let $\mathcal{P}$ be a Poisson pencil on a manifold $M$. We say that a singular distribution distribution $\mathcal{D}$ in the cotangent bundle $T^*M$ is \textbf{complete} if there exists an open dense $U \subset M$ such that $\mathcal{D}$ is a (smooth) regular distribution on $U$ and for any point $x \in U$ the following two conditions hold:

\begin{enumerate}

\item $\mathcal{D}$ is bi-isotropic, i.e. \[\mathcal{P}(x) \bigr|_{\mathcal{D}_x} \equiv 0.\] 

\item The dimension of $\mathcal{D}$ is \begin{equation} \label{Eq:DimCompleteCond} \dim \mathcal{D}_x = \dim M - \frac{1}{2}\operatorname{rk} \mathcal{P}(x).\end{equation}

\end{enumerate}

\end{definition}

\begin{theorem} \label{Th:CompCrit} Let $\mathcal{P}$ be a Poisson pencil on $M$ such that such that \[ \operatorname{rk} \mathcal{P}(x) = \operatorname{const}, \qquad \deg p_{\mathcal{P}}(\lambda) = \operatorname{const} \] on $M$. Then for any point $x_0 \in M$ the following conditions are equivalent:

\begin{enumerate}

\item \label{C:1} The extended core distribution $\hat{\mathcal{K}}$ is complete in a neighborhood of $x_0$.

\item \label{C:2}  The extended core distribution has maximal possible dimension at $x_0$: \begin{equation} \label{Eq:DimExtPointCrit} \dim \hat{\mathcal{K}}_{x_0} = \dim M - \frac{1}{2} \operatorname{rk} \mathcal{P}_{x_0}.\end{equation}

\item \label{C:3}  In the Jordan--Kronecker decomposition of $\mathcal{P}(x_0)$
all Jordan blocks are $2 \times 2$, all  eigenvalues $\lambda_1(x_0), \dots, \lambda_{N}(x_0)$ are distinct and \[ d \lambda_j (x_0) \not \in \mathcal{K}^{\mathbb{C}}_{x_0}.\] 
\end{enumerate}
    
\end{theorem}

\begin{proof}[Proof of Theorem~\ref{Th:CompCrit}] Obviously, under the conditions of the theorem, $\dim \hat{\mathcal{K}}$ is a  lower semicontinuous  functions. Thus, conditions \ref{C:1} and \ref{C:2} are equivalent by  Lemma~\ref{L:ExtCoreLemma}. Conditions \ref{C:2} and \ref{C:3}  are equivalent by Proposition~\ref{Prop:DimExtCore}. Theorem~\ref{Th:CompCrit}  is proved. \end{proof}

\begin{remark} In the analytic case if, roughly speaking, some equality holds locally, then it also holds globally. Thus, if $M$ is real or complex analytic and Theorem~\ref{Th:CompCrit} is satisfied at a point $x_0$ such that $\operatorname{rk}\mathcal{P}(x_0) = \operatorname{rk} \mathcal{P}$, then the extended core distribution $\hat{\mathcal{K}}$ is complete on $M$.
\end{remark}

\subsection{Jordan and Kronecker cases}

We say that a Poisson pencil $\mathcal{P}$ on a manifold $M$ is of \textbf{Kronecker type} (of \textbf{Jordan type}) if at a generic point $x \in M$ the Jordan--Kronecker decomposition of $\mathcal{P}(x)$ has only Kronecker blocks (respectively, only Jordan blocks). Note that Theorem~\ref{Th:CompCrit} holds for Poisson pencils of Kronecker and Jordan types. In the Kronecker case $\hat{\mathcal{K}} = \mathcal{K}$ and we get a complete family of functions in bi-involution. 

\begin{corollary} If $\mathcal{P}$ is a Poisson pencil on $M$ of Kronecker type, then the core distribution $\mathcal{K}$ is complete on $M$. \end{corollary}

In the Jordan case, roughly speaking, $\hat{\mathcal{K}}$ is complete iff all eigenvalues are distinct and locally non-constant. 

\begin{corollary}  Let $\mathcal{P}$ be a Poisson pencil on $M$ of Jordan type. The extended core distribution $\hat{\mathcal{K}}$ is complete on $M$ if and only if  in the Jordan--Kronecker decomposition of $\mathcal{P}(x_0)$ of a generic point $x_0 \in M$
all Jordan blocks are $2 \times 2$, all  eigenvalues $\lambda_1(x_0), \dots, \lambda_{N}(x_0)$ are distinct and $d \lambda_j(x_0) \not = 0$.
\end{corollary}

\subsection{Applications of the criterion}

 Let us briefly discuss how Theorem~\ref{Th:CompCrit} can be used in practice. First, we can determine if the Jordan blocks are $2\times 2$ using the following trivial statement about the characteristic polynomial.

\begin{proposition} Let $\mathcal{L} = \left\{A + \lambda B\right\}$ be a linear pencil and $p_{\mathcal{L}}$ be its characteristic polynomial. Then the following conditions are equivalent:

\begin{enumerate}

\item In the Jordan--Kronecker decomposition of $\mathcal{L}$ all Jordan blocks are $2 \times 2$ and have distinct eigenvalues. 

\item All roots of $p_{\mathcal{L}}$ are distinct.

\item In the decomposition of $p_{\mathcal{L}}$ into irreducible factors \[p_{\mathcal{L}} = f_1^{k_1}  \cdot \dots \cdot f_t^{k_t},  \] all degrees $k_i = 1$.

\end{enumerate}

\end{proposition}

Second, it may be easier to check that the conditions of Theorem~\ref{Th:CompCrit} are not satisfied. Then $\hat{\mathcal{K}}$ is not complete. In particular, we get the following.

\begin{corollary} \label{Cor:NotComp} Let $\mathcal{P}$ be a Poisson pencil on $M$. Assume that for a generic point $x \in M$ there are Jordan blocks in the Jordan--Kronecker decomposition of $\mathcal{P}(x)$ but they are not $2 \times 2$ Jordan blocks with distinct eigenvalues. Then the extended core distribution $\hat{\mathcal{K}}$ is not complete. 
\end{corollary}

\section{Lie algebras} \label{S:LieAlg}

In this section we briefly discuss how the results of this paper can be applied for some well-known commutative subalgebras of commutative subalgebras in the symmetric algebra $S(\mathfrak{g})$ of a finite-dimensional Lie algebra $\mathfrak{g}$.

Let $\mathfrak{g}$ be a finite-dimensional Lie algebra and $\mathfrak{g}^*$ be its dual space. There are two natural Poisson brackets on $\mathfrak{g}^*$:

\begin{enumerate}

    \item Linear \textbf{Lie--Poisson} bracket $\mathcal{A}_x$, given by \[ \left\{ f, g\right\} = \langle x, [df(x), dg(x) ] \rangle, \quad x \in \mathfrak{g}^*, \quad f,g:\mathfrak{g}^* \to \mathbb{C};\]

    \item Take $a \in \mathfrak{g}^*$. The corresponding constant  Poisson bracket $\mathcal{A}_a$ (so called \textbf{"frozen argument" bracket}) is given by \[ \left\{ f, g\right\} = \langle a, [df(x), dg(x) ] \rangle, \quad x \in \mathfrak{g}^*, \quad f,g:\mathfrak{g}^* \to \mathbb{C};\]
    
\end{enumerate}

For any $a \in \mathfrak{g}^*$ the brackets $\mathcal{A}_x$ and $\mathcal{A}_a$ are compatible. Their matrices at a point $x\in \mathfrak{g}^*$ are \begin{equation} \label{Eq:PencilMatLie} \mathcal{A}_x  = \left(\sum_k c_{ij}^k x_k \right), \qquad \mathcal{A}_a = \left(\sum_k c_{ij}^k a_k \right). \end{equation}

In \cite{BolsZhang} \textbf{Jordan--Kronecker invariants} of $\mathfrak{g}$ were introduced. Roughly speaking, these invariants are sizes of Kronecker blocks and sizes of Jordan blocks grouped by eigenvalues for a generic linear pencil $\mathcal{P}_{x,a} = \mathcal{A}_x + \lambda \mathcal{A}_a$, given by \eqref{Eq:PencilMatLie}. It was also shown that 
characteristic polynomial for the pencil $\mathcal{P}_{x,a}$ is related with the \textbf{fundamental semi-invariant} $p_{\mathfrak{g}}$ of $\mathfrak{g}$. Namely, that \[p_{\mathcal{P}_{x,a}}(\lambda) = p_{\mathfrak{g}}(x + \lambda a). \]

We discuss the following commutative subalgebras of the symmetric algebra $S(\mathfrak{g})$:

\begin{itemize}
    \item the \textbf{algebra of polynomial shifts}
$\mathcal{F}_a$ (for the definition and details see  \cite{BolsZhang});  \item the \textbf{extended Mischenko-Fomenko subalgebras} $\tilde{\mathcal{F}}_a$ (see \cite{Izosimov14}); 
    \item and the \textbf{algebra of shift of semi-invariants} $\mathcal{F}_{a}^{\textrm{si}}$ (see  \cite{Kozlov23Shift}).

\end{itemize}

Denote by
 \[ d\mathcal{F}_a = \operatorname{span} \left\{df(x), f \in \mathcal{F}_a \right\}\] the distribution "spanned" in $T^*\mathfrak{g}^*$ by $\mathcal{F}_a$. We also consider similar distributions $d\tilde{\mathcal{F}}_a, d\mathcal{F}_a^{\operatorname{si}}$ for $\tilde{\mathcal{F}}_a$ and $\mathcal{F}_a^{\operatorname{si}}$ respectively. For a regular $a \in \mathfrak{g}^*$ the followings fact about this distributions are known:

\begin{itemize}

\item In \cite[Section 5]{BolsZhang} it was explained that the algebra of polynomial shifts $\mathcal{F}_a$ "spans" the core distribution $\mathcal{K}$ almost everywhere.

\item By \cite[Proposition 5.1]{Izosimov14} the extended Mischenko-Fomenko subalgebra $\mathcal{F}_a$ "spans" the extended core distribution $\hat{\mathcal{K}}$  almost everywhere.

\item Also, in \cite{Kozlov23Shift} it was shown that for the algebra of shifts of semi-invariants $ d\mathcal{F}_a^{\operatorname{si}} = d\tilde{\mathcal{F}}_a$ almost everywhere.

\end{itemize}

Formally, we have the following.

\begin{proposition} Let $a \in \mathfrak{g}^*$ be a regular element. Then for a generic $x\in \mathfrak{g}^*$ we have the following:

\begin{enumerate}

\item  $d\mathcal{F}_a (x) = \mathcal{K} (x) $;

\item  $d\tilde{\mathcal{F}}_a (x) = \mathcal{F}_a^{\operatorname{si}} = \hat{\mathcal{K}} (x) $.

\end{enumerate}

\end{proposition}

For the Lie algebras the completeness criterion, given by Theorem~\ref{Th:CompCrit}, can be reformulated as follows.

\begin{theorem} \label{T:Jordan1Form} Let $a \in \mathfrak{g}^*$ be a regular element. The extended Mischenko-Fomenko subalgebra $\tilde{\mathcal{F}}_a$ is complete if and only if the following $2$ conditions hold. 

\begin{enumerate}

\item The Jordan--Kronecker decomposition of a generic pencil $\left\{\mathcal{A}_x + \mathcal{A}_a\right\}$ contains just one trivial $2 \times 2$ Jordan $\lambda_i$-block for each root $\lambda_i$ of $p_{\mathfrak{g}}(x - \lambda a) =0$.

\item \label{Item:Cond2Jord} Each root $\lambda_i$ is functionally independent with the generators of $\mathcal{F}_a$, i.e. \begin{equation} \label{Eq:DLdFNotIn} d \lambda_i(x) \not \in d \mathcal{F}_a(x)\end{equation} on an open dense subset of $\mathfrak{g}^*$.

\end{enumerate}

\end{theorem}

\begin{remark}
In terms of \cite{Kozlov23JKRealization}, condition \eqref{Eq:DLdFNotIn} means that each $\lambda_i$ is not a core characteristic number. \end{remark}

\begin{remark} Note that we can immediately say that $\tilde{\mathcal{F}}_a$  is not complete if in the Jordan--Kronecker invariants of $\mathfrak{g}$ not all Jordan blocks are $2 \times 2$ with distinct eigenvalues (cmp. Corollary~\ref{Cor:NotComp}). \end{remark}

In \cite{Izosimov14} the following completeness criterion for the  extended Mischenko-Fomenko subalgebras $\tilde{\mathcal{F}}_a$ was described. 

\begin{itemize}
\item Let $\operatorname{Sing} \subset \mathfrak{g}^*$ be the set of singular elements and $\operatorname{Sing}_0$ be the union of all irreducible components of $\operatorname{Sing}$ that have dimension $\dim \mathfrak{g} -1$. (If $\operatorname{codim} \operatorname{Sing} \geq 2$, then $\operatorname{Sing}_0 = \emptyset$.)

\item Consider the subset
\[ \operatorname{Sing}_b = \left\{ y \in \operatorname{Sing}_0 \, \, \bigr| \,\, \mathfrak{g}_y \simeq  \operatorname{aff}(1) \oplus \mathbb{C}^{\operatorname{ind} \mathfrak{g}} \right\} \subset \operatorname{Sing}_0, \]
 where $\operatorname{aff}(1)$  is the 2-dimensional non-abelian Lie algebra and \[\mathfrak{g}_y = \left\{\xi \in \mathfrak{g} \,\, \bigr| \,\, \operatorname{ad}^*_{\xi} (x) = 0 \right\}.\] 
\end{itemize} 
 
 \begin{theorem}[\cite{Izosimov14}] \label{T:IzosTh} Let $\mathfrak{g}$  be a finite-dimensional complex
Lie algebra and $a \in \mathfrak{g}^*$ be a regular element. The extended Mischenko-Fomenko
subalgebra $\tilde{\mathcal{F}}_a$ is complete if and only if $\operatorname{Sing}_b$ is dense in $\operatorname{Sing}_0$.
\end{theorem}

It was explained in \cite[Section 7]{BolsZhang} why the criteria from Theorems~\ref{T:Jordan1Form} and \ref{T:IzosTh} are equivalent. It also follows from the results from \cite[Section 10.3.2]{Kozlov23JKRealization}.

\end{document}